\documentclass[a4paper, 12pt]{article}
\usepackage{amsthm,amssymb,amsmath,txfonts,textcomp}
\usepackage{mathrsfs}
\usepackage{amsfonts}
\usepackage[mathscr]{eucal}
\usepackage{color}
\usepackage{verbatim}
\usepackage[colorlinks,linkcolor=red,anchorcolor=blue,citecolor=blue]{hyperref}
\usepackage{cite}
\usepackage{float}
\usepackage{indentfirst}
\newtheorem{thm}{Theorem}[section]
\newtheorem{lemma}{Lemma}[section]

\newtheorem{rem}{Remark}[section]

\allowdisplaybreaks

\numberwithin{equation}{section}

\begin{document}
\title{On the finiteness of the Morse index of self-shrinkers}
\author{\textsc{Xu-Yong Jiang $\quad$ He-Jun Sun$^\dag$   $\quad$ Peibiao Zhao }}

\date{}

\maketitle

{\narrower
\vskip3mm \noindent {\it\bfseries Abstract}:
In this paper, we present a sufficient condition for finite Morse index of complete properly self-shrinkers.
We prove that a complete properly embedded self-shrinker in $\mathbb{R}^{n+1}$
with  finite asymptotically conical ends or  asymptotically cylindrical ends must have finite Morse index.
Moreover, as an application of this result, we show that a complete properly embedded self-shrinker in $\mathbb{R}^3$ with finite genus has finite Morse index.



}

\footnotetext{
$^\dag$  Corresponding author: He-Jun Sun

   This paper was supported by the National Natural Science Foundation of China (Grant Nos.11001130, 11871275) and the Fundamental Research Funds for the Central
Universities (Grant No. 30917011335).}

\section{Introduction}

Self-shrinkers are a special class of solutions
of mean curvature flow in which each time slice is a rescaling of itself.
More precisely, a hypersurface $\Sigma$ in $\mathbb{R}^{n+1}$ is called a self-shrinker if it satisfies
\begin{align}\label{SS}
H=-\frac{1}{2}\langle \mathbf{x}, \mathbf{n}\rangle ,
\end{align}
where $H=-\mathrm{div}(\mathbf{n})$ is the mean curvature of $\Sigma$, $\mathbf{x}$ is the position vector on $\Sigma$ and $\mathbf{n}$ is
the outward unit normal vector of $\Sigma$ at $\mathbf{x}$.
One of important problems about the mean
curvature flows is to understand the singularity that the flow goes through.
Self-shrinkers play a key role in
the study of the mean curvature flow because they describe all possible blowups
at a given singularity.

It is well known that self-shrinkers in $\mathbb{R}^{n+1}$ can be viewed as
critical points of the weighted volume functional
\begin{align}
F(\Sigma)=\int_\Sigma e^{-\frac{|\mathbf{x}|^2}{4}} d\mu,
\end{align}
where $d\mu$ denotes the Riemannian measure associated to the induced metric on $\Sigma$.
In fact, let $\Sigma_t$ be a normal variation of $\Sigma$,
denote $f\mathbf{n}$ be a compact supported variation vector field on $\Sigma$,
then we have
\begin{align}
\frac{d}{dt}\Bigg|_{t=0}F(\Sigma_t)=-\int_\Sigma f\left(H+\langle\mathbf{x},\mathbf{n}\rangle\right)e^{-\frac{|\mathbf{x}|^2}{4}}d\mu.
\end{align}
One can find that $\Sigma$ is a critical point for the weighted volume functional
if and only if it is a self-shrinker.
It is natural to ask whether a self-shrinker is a local minimum of the weighted volume functional.
Denote by $A$ the second fundamental form of $\Sigma$.
Jacobi operator $L$ is defined by
\begin{align}
Lf=\Delta f - \frac{1}{2}\langle \mathbf{x}, \nabla f\rangle + |A|^2 +\frac{1}{2}.
\end{align}
In order to answer the above question, we need the following second variation formula (cf \cite{CM12}):
\begin{align}
I(f,f):=\frac{d^2}{dt^2}\Bigg|_{t=0}  F(\Sigma_t)   =-\int_\Sigma fLf e^{-\frac{|\mathbf{x}|^2}{4}} d\mu,
\end{align}

The theory of self-adjoint elliptic operators tells us that the
eigenvalues of $L$ on a relatively compact domain $\Omega\subset\Sigma$ are
\begin{align}
\lambda_1\le \lambda_2\le\cdots \to \infty.
\end{align}
Define the Morse index $\text{Ind}(\Omega)$ as the number of negative eigenvalues
of $L$ counted with the multiplicity. The domain monotonicity of eigenvalues
implies that if $\Omega\subset\Omega'$, then $\text{Ind}(\Omega)\le \text{Ind}(\Omega')$.
Hence the Morse index of $\Sigma$ is defined by
\begin{align}
\text{Ind}(\Sigma)=\sup\left\{ \text{Ind}(\Omega)\big| \Omega\subset\subset \Sigma \right\} .
\end{align}
It is obvious that $\text{Ind}(\Sigma)$ is always finite if $\Sigma$ is compact and
$\text{Ind}(\Sigma)$ can be infinite if $\Sigma$ is non-compact.
We say that $\Omega\subset\Sigma$ is stable if $\text{Ind}(\Omega)=0$,
namely $I(f,f)\ge 0$ holds for all $f\in C^\infty_0(\Omega)$.
In the instability case, the Morse index of self-shrinkers measures
the number of linearly independent directions that decrease the weighted volume up to second order.

It is an important problem to estimate the Morse index of self-shrinkers (cf. \cite{JSZ2019}).
In 2012, Colding and Minicozzi\cite{CM12} proved that the Morse index of properly
self-shrinkers is greater than or equal to 1.
In 2017, Barbosa, Sharp and Wei\cite{BSW17} gave an upper bound for the Morse index
of closed embedded self-shrinkers in terms of their entropy, genus and maximum radii.
In 2018, Impera, Rimoldi and Savo\cite{IRS19} gave a lower bound for the Morse index of self-shrinkers
in terms of their genus. Sarouis Aiex\cite{SA18} improved this lower bound.
In 2019, Impera\cite{Im19} derived some gap theorems for the Morse index of complete properly immersed self-shrinkers.
Recently, Berchenko and Kogan\cite{BK20} gave upper and lower bounds for the Morse index of rotationally
symmetric self-shrinking tori in terms of their entropy, maximum and minimum radii.
Moreover, their results gave some bounds for the Morse index of Angenent torus.

It is a complicated problem to investigate the Morse index of non-compact self-shrinkers.
As far as we know, there is no finiteness results on the Morse index
for general non-compact self-shrinkers.
In this paper, we first obtain the following result for complete properly embedded self-shrinkers in $\mathbb{R}^{n+1}$ with  finite asymptotically conical ends or  asymptotically cylindrical ends.

\begin{thm} \label{thm.1}
Let $\Sigma$ be a complete properly embedded self-shrinker in $\mathbb{R}^{n+1}$
with  finite asymptotically conical ends or  asymptotically cylindrical ends,
then $\Sigma$ has finite Morse index.
\end{thm}

\begin{rem}
The proof of Theorem 1.1 leans on the classification for self-shrinkers' ends with
finite genus given in \cite{Wa16b,SW20}.
In \cite{Wa16b}, Wang proved that all ends of self-shrinking surfaces in $\mathbb{R}^3$ with finite genus
are either asymptotically conical or cylindrical.
It is not clear whether similar results holds on higher dimension.
In view of Wang's results, the assumptions of Theorem 1.1 may be not particularly strong.
\end{rem}

\begin{rem}
It is conjectured that the only self-shrinker with
asymptotically cylindrical ends is the generalized cylinder (cf. \cite{Wa16}).
As we know, the Morse index of the generalized cylinder is $n+2$.
Thus, if the conjecture will be proved to be true, then we don't need to consider the asymptotically cylindrical case in the proof of Theorem 1.1.
\end{rem}

Moreover, we consider self-shrinkers in $\mathbb{R}^3$.
In fact, two-dimensional self-shrinking surfaces in $\mathbb{R}^3$ has attracted much attention.
When $\Sigma$ is a non-compact properly embedded self-shrinker in $\mathbb{R}^3$
of finite topology, Wang\cite{Wa16b} showed that each ends of $\Sigma$
are either asymptotically conical or asymptotically cylindrical.
Though Wang stated the result for finite topology, Sun and Wang \cite{SW20} pointed out that
Wang's result is still true for self-shrinkers with finite genus.

As an application of Theorem \ref{thm.1}, the following result can be derived by using Theorem 1.1 in \cite{Wa16b}, Theorem A.1, Lemma A.2 in \cite{SW20} and Theorem \ref{thm.1}.

\begin{thm}\label{thm.2}
Let $\Sigma$ be a complete properly embedded self-shrinker in $\mathbb{R}^3$ with finite genus,
then $\Sigma$ has finite Morse index.
\end{thm}

\begin{rem}
Minimal surfaces and self-shrinkers are two kinds of geometric objects with many
similar geometric properties. It is well known that a complete minimal surface in $\mathbb{R}^3$
has finite Morse index if and only if it has finite total curvature. Indeed,
if $\Sigma$ is a complete minimal surface, there exists an absolute constant $C$ such that
\begin{align*}
\text{Ind}(\Sigma)\le C\int_\Sigma -K d\mu,
\end{align*}
where $K=-\frac{1}{2}|A|^2$ is the Gauss curvature of $\Sigma$.
There have been a number of results that relate the Morse index to the total curvature (for example, see \cite{GNY04} and the references therein).
The results of this paper inspire us to ask  whether there is an absolute constant $C$ such that
\begin{align*}
\text{Ind}(\Sigma)\le C\int_\Sigma (|A|^2+1)e^{-\frac{|x|^2}{4}} d\mu,
\end{align*}
holds for a complete properly embedded self-shrinker $\Sigma$.
It is an interesting question that deserves further consideration.
\end{rem}
\begin{rem}
Recently, we learned about the work \cite{ANZ21} of  Alencar, Neto and Zhou. They proved that
a self-shrinker in $\mathbb{R}^3$ with finite Morse index must be proper and with finite topology.
Combing with Theorem \ref{thm.2} of our paper, one can find that
a complete embedded self-shrinker in $\mathbb{R}^3$ has finite Morse index if and only if it was proper and with finite topology.
\end{rem}

\section{Proof of Theorem \ref{thm.1}}

Here we first give some notations and definitions which will be used.
Denote by $ B_r$  the open ball in $\mathbb{R}^{n+1}$ centered at $O$ with radius
$r$, where $O$ is the origin of $\mathbb{R}^{n+1}$.
We say that $\mathcal{C}\subset \mathbb{R}^{n+1}$ is a regular cone with vertex at $O$ generated by $\Gamma$,
if
\begin{align}
\mathcal{C}=\{l\Gamma: 0\le l < \infty \}
\end{align}
where $\Gamma$ is a smooth closed embedded submanifold of unit sphere $\mathbb{S}^n$ with codimension one.
A hypersurface $\Sigma \subset \mathbb{R}^{n+1}$ is said to be asymptotically conical
if there exists a regular cone $\mathcal{C}$ with vertex at $O$ such that $\lambda \Sigma$ converges smoothly in
compact sets to $\mathcal{C}$ as $\lambda\to 0^+$. That is to say, for any $r>0$,
$\lambda \Sigma \cap \{ B_r\backslash B_{1/r}\}$ converges to
$\mathcal{C} \cap \{ B_r\backslash B_{1/r}\}$ in the $C^k$ topology as $\lambda\to 0^+$.
Similarly, a hypersurface $\Sigma \subset \mathbb{R}^{n+1}$ is said to be asymptotically cylindrical
if exists a unit vector $\xi\in \mathbb{R}^{n+1}$ such that $\Sigma-\lambda \xi$ converges smoothly
in compact sets to the generalized cylinder
$\mathbb{S}^k({\sqrt{2k}})\times \mathbb{R}^{n-k}$
as $\lambda\to\infty$.

Now we establish two integral inequalities for regular cones and generalized cylinders in Lemmas \ref{lemma.A} and \ref{lemma.D}, which will paly roles in the proofs of Lemmas \ref{lemma.C} and \ref{lemma.E}. We first consider regular cones.

\begin{lemma} \label{lemma.A}
Let $\mathcal{C} \subset \mathbb{R}^{n+1}$ be a regular cone with vertex at $O$.
For any given positive constant $C$, there exists a positive constant
$r$ such that for all $f\in C^\infty_0(\mathcal{C}\backslash B_{r})$,
\begin{equation}\label{lemma3.1}
\int_{\mathcal{C}\backslash B_{r}}|\nabla f(\mathbf{x})|^2 e^{-\frac{|\mathbf{x}|^2}{4}} d\mu
\geq
C\int_{\mathcal{C}\backslash B_{r}} f^2(\mathbf{x}) e^{-\frac{|\mathbf{x}|^2}{4}}d\mu,
\end{equation}
where $\nabla$ is the gradient operator on $\mathcal{C}$.
\end{lemma}

\begin{proof}
Suppose that $\mathcal{C}$ is a regular cone with vertex at $O$ generated by $\Gamma$,
where $\Gamma$ is a smooth closed embedded submanifold of $n$-dimensional unit sphere $\mathbb{S}^n$ with codimension one.
Then we use polar coordinates $(\rho, \eta)$, where $\rho=|\mathbf{x}|$ and $\eta\in\Gamma$.
For any $f\in C^\infty_0(\mathcal{C}\backslash B_r)$, $r>1$, set
\begin{equation}
f(\mathbf{x})=\int^{\rho}_r \frac{\partial f(s,\eta)}{\partial s} ds.
\end{equation}
Using Holder inequality, we deduce
\begin{equation}
\begin{aligned}
f^2(\mathbf{x}) &\le \Bigg( \int^\rho_r \Big|\frac{\partial f(s,\eta)}{\partial s} \Big| ds\Bigg)^2 \\
&\le \int^\rho_r \Big|\frac{\partial f(s,\eta)}{\partial s} \Big|^2 e^{-\frac{s^2}{4}}s^{n-1}d s
	\int^\rho_r e^{\frac{s^2}{4}} s^{1-n} d s \\
&\le \int^\infty_r |\nabla f( s,\eta)|^2 e^{-\frac{s^2}{4}}s^{n-1}ds
	\int^\rho_r e^{\frac{ s^2}{4}}\rho^{1-n} d s .
\end{aligned}
\end{equation}
Hence we have
\begin{equation}\label{r-0}
\begin{aligned}
\int_{\mathcal{C}\backslash B_r} f^2(\mathbf{x}) e^{-\frac{|\mathbf{\mathbf{x}}|^2}{4}}d\mu
&=\int_\Gamma \int^\infty_r f^2(\rho,\eta)e^{-\frac{\rho^2}{4}} \rho^{n-1}d\rho d\nu \\
& \leq \int^\infty_r \Bigg(\int_\Gamma\int^\infty_r |\nabla f(s,\eta)|^2 e^{-\frac{s^2}{4}}s^{n-1}ds d\nu \Bigg)
	\Bigg(\int^\rho_r e^{\frac{s^2}{4}}s^{1-n} ds\Bigg) e^{-\frac{\rho^2}{4}} \rho^{n-1}d\rho \\
&= \Bigg(\int_{\mathcal{C}\backslash B_r}|\nabla f(\mathbf{x})|^2 e^{-\frac{|\mathbf{x}|^2}{4}} d\mu\Bigg)
\Bigg[\int^\infty_r \Big(\int^\rho_r e^{\frac{s^2}{4}}s^{1-n} ds\Big) e^{-\frac{\rho^2}{4}} \rho^{n-1}d\rho\Bigg].
\end{aligned}
\end{equation}
It is easy to get
\begin{equation}
\lim_{\rho\to\infty} \frac{\int^\rho_r e^{\frac{s^2}{4}}s^{1-n} ds}{ e^{\frac{\rho^2}{4}} \rho^{3-n}}=0.
\end{equation}
It means that there exists a positive constant $r_1$ such that for $ r \geq r_1$,
\begin{equation}\label{r-1}
\frac{\int^\rho_r e^{\frac{s^2}{4}}s^{1-n} ds}{ e^{\frac{\rho^2}{4}} \rho^{1-n}} \le \rho^{-2}.
\end{equation}
Moreover, since $\rho^{-2}$ is integrable at infinity,  there exists $r_2$ such that for $ r \geq r_2$,
\begin{equation}\label{r-2}
\int^\infty_{r}\rho^{-2} d\rho\leq \frac{1}{C}.
\end{equation}
Therefore, combining  \eqref{r-1} with \eqref{r-2}, we know that when $r \geq \max\{r_1,r_2\}$,
\begin{equation}\label{r-3}
\int^\infty_r \big(\int^\rho_r e^{\frac{s^2}{4}}s^{1-n} ds\big) e^{-\frac{\rho^2}{4}} \rho^{n-1}d\rho \leq \frac{1}{C}.
\end{equation}
Using \eqref{r-0} and \eqref{r-3}, we can get \eqref{lemma3.1}.
This completes the proof of  Lemma \ref{lemma.A}.
\end{proof}

Using cylindrical coordinates and the same method as Lemma \ref{lemma.A}, we can obtain the following result for  generalized cylinders.

\begin{lemma} \label{lemma.D}
Let $\mathbb{S}^k({\sqrt{2k}})\times \mathbb{R}^{n-k}$ be a generalized cylinder. For any given positive constant $C$, there exists
a positive constant $r$ such that for all
$f\in C^\infty_0\big((\mathbb{S}^k({\sqrt{2k}})\times \mathbb{R}^{n-k})\backslash B_{r}\big)$,
\begin{equation}
\int_{\left(\mathbb{S}^k({\sqrt{2k}})\times \mathbb{R}^{n-k}\right)\backslash B_{r}}|\nabla f(\mathbf{x})|^2 e^{-\frac{|\mathbf{x}|^2}{4}} d\mu
\geq
C\int_{\left(\mathbb{S}^k({\sqrt{2k}})\times \mathbb{R}^{n-k}\right)\backslash B_{r}} f^2(\mathbf{x}) e^{-\frac{|\mathbf{x}|^2}{4}}d\mu,
\end{equation}
where $\nabla$ is the gradient operator on $\mathbb{S}^k({\sqrt{2k}})\times \mathbb{R}^{n-k}$.
\end{lemma}

For an asymptotically conical end of self-shrinkers, Wang \cite{Wa14} proved that
it can be given by the graph of a smooth function over a regular cone.
Moreover, Wang gave some gradient estimates of function defined on the
tangent spaces of regular cone with respect to the Euclidean metric (see Lemma 2.2 of \cite{Wa14}).
Here we give some sharper gradient estimates of functions defined on the regular cone with respect to the intrinsic metric in Lemma \ref{lemma.B}.
In some sense, our results are slightly different from results of Wang.

\begin{lemma} \label{lemma.B}
Let $\Sigma$ be an end of a smooth properly embedded self-shrinker in $\mathbb{R}^{n+1}$,
which is asymptotically conical to a regular cone $\mathcal{C}$.
There exists a positive constant $r$ such that, out of a compact set $K$, $\Sigma$ is given by the
graph of a smooth function $u(\cdot):\mathcal{C}\backslash B_{r}\to \mathbb{R}$.
Moreover, there exists a positive constants $C$  such that for $\mathbf{z}\in \mathcal{C}\backslash B_{R}$,
\begin{equation}\label{eq2.a}
|u(\mathbf{z})|\le C|\mathbf{z}|^{-1}
\end{equation}
and
\begin{equation} \label{eq2.b}
|\nabla u(\mathbf{z})| \le C|\mathbf{z}|^{-2},
\end{equation}
where $\nabla$ is the gradient operator on $\mathcal{C}$.
\end{lemma}

\begin{proof}
Let $\Sigma_t=\sqrt{t}\Sigma, 0<t\le 1$.
According to  Lemma 2.2 in \cite{Wa14},
there exists a positive constant $r_1$ such that, out of compact subset $K_t$,
$\Sigma_t$ is given by the graph of a smooth function
$v(\cdot,t):\mathcal{C}\backslash B_{r_1}\to \mathbb{R}$.
Moreover, for any $\mathbf{z}_0\in \mathcal{C}\backslash B_{r_1}$, if $\Sigma_t$ is locally
written as the graph of a smooth function $\tilde{v}(\cdot,t)$ on $B_1(\mathbf{z}_0)\cap T_{\mathbf{z}_0}\mathcal{C}$,
there exists a positive constant $C_1$ such that the following uniform estimates hold
\begin{equation} \label{eq2.1}
|D \tilde{v}|\le C_1,
\end{equation}
\begin{equation} \label{eq2.1-a}
 |\partial_t \tilde{v}|\le C_1|\mathbf{z}_0|^{-1}
\end{equation}
and
\begin{equation} \label{eq2.1-b}
    |D\partial_t \tilde{v}|\le C_1|\mathbf{z}_0|^{-2}.
\end{equation}
Here $D$ is the Euclidean gradient on $T_{\mathbf{z}_0}\mathcal{C}$, and $\partial_t$ denotes the
partial derivative with respect to $t$ at the points in $T_{\mathbf{z}_0}\mathcal{C}$.
Set $u(\mathbf{z})=v(\mathbf{z},1)$.
Hence there exists a positive constant $r_1$ such that, out of a compact set $K$, $\Sigma$ is given by the
graph of a smooth function $u(\cdot):\mathcal{C}\backslash B_{r_1}\to \mathbb{R}$.

Since $\Sigma$ is asymptotically conical to a regular cone $\mathcal{C}$, we have
\begin{equation} \label{eq2.3}
\lim_{t\to 0^+}|v|=0   \quad \mbox{and}  \quad \lim_{t\to 0^+}|\nabla v|=0.
\end{equation}
Moreover, since $v(\mathbf{z}_0,t)=\tilde{v}(\mathbf{z}_0,t)$, we get
\begin{equation}
|u(\mathbf{z}_0)| = \Big|\int_0^1 \partial_t \tilde{v}(\mathbf{z}_0,t) dt\Big| \le \int_0^1 |\partial_t \tilde{v}(\mathbf{z}_0,t)|dt \le C_1|\mathbf{z}_0|^{-1}.
\end{equation}
That is to say, \eqref{eq2.a} is true.

It is obvious that there exists a positive constant $r_2$ such that
for $\mathbf{z}_0\in \mathcal{C}\backslash B_{r_2}$,
$\mathcal{C}$ can be locally written as the graph of a smooth function
$\tilde{w}$ on $B_1(\mathbf{z}_0)\cap T_{\mathbf{z}_0}\mathcal{C}$,
and there exists a positive constant $C_2$ such that
\begin{equation} \label{eq2.2}
|D^2 \tilde{w}|\le C_2|\mathbf{z}_0|^{-1}.
\end{equation}

We parametrize $T_{\mathbf{z}_0}\mathcal{C}$ by
\begin{equation*}
\begin{aligned}
&F:\mathbb{R}^n\to T_{\mathbf{z}_0}\mathcal{C}\\
&F(p)=\mathbf{z}_0+\sum_i p_i \mathbf{e}_i,
\end{aligned}
\end{equation*}
where $p=(p_1,\cdots,p_n)$, $\{\mathbf{e}_1,\cdots,\mathbf{e}_n\}$
is an orthonormal basis of $T_{\mathbf{z}_0}\mathcal{C}-\mathbf{z}_0$, and $\mathbf{e}_{n+1}$
is the unit normal vector of $T_{\mathbf{z}_0}\mathcal{C}$.
For $p,q\in \mathbb{R}^n$, we respectively identify $\tilde{v}(p,t)$, $\tilde{w}(p)$ and $v(p,t)$
with $\tilde{v}(F(p),t)$, $\tilde{w}(F(p))$ and ${v}(F(p)+\tilde{w}(F(p))\mathbf{e}_{n+1},t)$ .
Note that the unit normal $\mathbf{n}$ of $\mathcal{C}$ at  point $F(p)+\tilde{w}(p)\mathbf{e}_{n+1}$
is given by
\begin{equation}
\mathbf{n}(p)=\sum_i \frac{-\partial_i \tilde{w}}{\sqrt{1+|D\tilde{w}|^2}}\mathbf{e}_i+\frac{1}{\sqrt{1+|D\tilde{w}|^2}}\mathbf{e}_{n+1}.
\end{equation}
Thus we have
\begin{equation}\label{eq2.c}
q_i =p_i-\frac{\partial_i \tilde{w}}{\sqrt{1+|D\tilde{w}|^2}}v
\end{equation}
and
\begin{equation} \label{eq2.d}
\tilde{v}(q,t)=\tilde{w}(p)+\frac{1}{\sqrt{1+|D\tilde{w}|^2}}v.
\end{equation}
Differentiating \eqref{eq2.c} and \eqref{eq2.d} with respect to $p_i$, we obtain
\begin{equation}\label{eq2.e}
 \partial_iq_j =\delta_{ji}-\frac{\partial_{ji}^2\tilde{w}}{\sqrt{1+|D\tilde{w}|^2}}v
	+\frac{\partial_j\tilde{w}}{(1+|D\tilde{w}|)^{3/2}}\bigg(\sum_l\partial_l\tilde{w}\partial_{li}^2\tilde{w} \bigg)v
	-\frac{\partial_j\tilde{w}}{\sqrt{1+|D\tilde{w}|^2}}\partial_i v
\end{equation}
and
\begin{equation} \label{eq2.f}
\sum_k\partial_k\tilde{v}\partial_i q_k = \partial_i\tilde{w}+\frac{\partial_i v}{\sqrt{1+|D\tilde{w}|^2}}
	-\frac{\sum_l \partial_l\tilde{w}\partial^2_{li}\tilde{w}}{(1+|D\tilde{w}|)^{3/2}}v .
\end{equation}
Then we deduce
\begin{equation}
\begin{aligned} \label{eq2.g}
\frac{1+\sum_k\partial_k\tilde{v}\partial_k\tilde{w}}{\sqrt{1+|D\tilde{w}|^2}}\partial_i v
=\partial_i\tilde{v}-\partial_i\tilde{w}-\frac{\sum_k\partial_k\tilde{v}\partial^2_{ki}\tilde{w}}{\sqrt{1+|D\tilde{w}|^2}}v
+\frac{1+\sum_k\partial_k\tilde{v}\partial_k\tilde{w}}{(1+|D\tilde{w}|)^{3/2}}\bigg(\sum_l\partial_l\tilde{w}\partial^2_{li}\tilde{w}\bigg) v .
\end{aligned}
\end{equation}
Calculating \eqref{eq2.g} at $p=0$, noticing that $D\tilde{w}(0)=0$ and $q(0)=0$, we know that at point $\mathbf{z}_0$,
\begin{equation} \label{eq2.h}
\partial_iv=\partial_i\tilde{v}-\bigg(\sum_k\partial_k\tilde{v}\partial^2_{ki}\tilde{w}\bigg) v.
\end{equation}
It is easy to find that the metric $g=(g_{ij})$ on $\mathcal{C}$  is determined by the coefficients
\begin{equation}
 g_{ij}=\delta_{ij}+\partial_i\tilde{w}\partial_j\tilde{w},
 \end{equation}
and the inverse  $g^{-1}=(g^{ij})$ of $g$ is determined by the coefficients
 \begin{equation}\label{eq2.27}
 g^{ij}=\delta_{ij}-\frac{\partial_i\tilde{w}\partial_j\tilde{w}}{1+|D\tilde{w}|^2}.
\end{equation}
Therefore, using \eqref{eq2.h} and \eqref{eq2.27}, we know that at point $\mathbf{z}_0$,
\begin{equation}
\begin{aligned} \label{eq2.i}
|\nabla v|^2 &=g^{ij}\partial_i v\partial_j v \\
&= |D\tilde{v}|^2+\sum_i\bigg(\sum_k\partial_k\tilde{v}\partial^2_{ki}\tilde{w}\bigg)^2 v^2
	-2\bigg(\sum_{ik} \partial_i\tilde{v}\partial_k\tilde{v}\partial^2_{ki}\tilde{w}\bigg) v.
\end{aligned}
\end{equation}

Let $t\to 0^+$ at point $\mathbf{z}_0$ in \eqref{eq2.i}. Then we can find that  the left hand side of \eqref{eq2.i} goes to $0$,
the second term and the third term on the right hand side of \eqref{eq2.i} also go to $0$ due to \eqref{eq2.1} and \eqref{eq2.3}.
Hence we obtain
\begin{equation}
\lim_{t\to 0^+} |D\tilde{v}(\mathbf{z}_0,t)|=0
\end{equation}
and
\begin{equation} \label{eq2.4}
|D\tilde{v}(\mathbf{z}_0,1)|\le \sqrt{n}\int_0^1\big|\partial_t D\tilde{v}(\mathbf{z}_0,t)\big|dt \le \sqrt{n}C_1|\mathbf{z}_0|^{-2}.
\end{equation}
Putting $t=1$ in \eqref{eq2.i}, we find that at point $\mathbf{z}_0$,
\begin{equation} \label{eq2.j}
|\nabla u|^2 = |D\tilde{v}|^2+\sum_i\bigg(\sum_k\partial_k\tilde{v}\partial^2_{ki}\tilde{w}\bigg)^2 u^2
	-2\bigg(\sum_{ik} \partial_i\tilde{v}\partial_k\tilde{v}\partial^2_{ki}\tilde{w}\bigg) u.
\end{equation}
From \eqref{eq2.2} and \eqref{eq2.4},
we can observe that all terms on the right hand side of \eqref{eq2.j}
are bounded by $|\mathbf{z}_0|^{-4}$ multiplied by a constant.
Thus  \eqref{eq2.b} is true. This concludes the proof of Lemma \ref{lemma.B}.
\end{proof}

Now we consider asymptotically conical ends and asymptotically cylindrical
ends of smooth properly embedded self-shrinkers in $ \mathbb{R}^{n+1}$ and establish two integral inequalities, which will play important role in the proof of Theorem \ref{thm.1}.

\begin{lemma} \label{lemma.C}
Let $\Sigma$ be an asymptotically
conical end of  a smooth properly embedded self-shrinker in $\mathbb{R}^{n+1}$.
For any given positive constant $C$, there exists a positive constant
$r$ such that for all $f\in C^\infty_0(\Sigma\backslash B_{r})$, we have
\begin{equation}\label{eq2.32}
C\int_{\Sigma\backslash B_{r}} f^2(\mathbf{x}) e^{-\frac{|\mathbf{x}|^2}{4}}d\mu
	 \le \int_{\Sigma\backslash B_{r}}|\nabla f(\mathbf{x})|^2 e^{-\frac{|\mathbf{x}|^2}{4}} d\mu,
\end{equation}
where $\nabla$ is the gradient operator on $\Sigma$.
\end{lemma}

\begin{proof}
Assume that $\Sigma$ is asymptotically conical to a regular cone $\mathcal{C}$.
It follows form Lemma \ref{lemma.B} that there exists a positive constant $r_1$ such that, outside of a compact set $K$,
$\Sigma$ is given by the graph of function $u:\mathcal{C}\backslash B_{r_1}\to \mathbb{R}$.
Hence we define a differential isomorphism $\Pi:\mathcal{C}\backslash B_{r_1}\to \Sigma\backslash K$
as
$$\Pi(\mathbf{z})=\mathbf{z}+u(\mathbf{z})\mathbf{n}(\mathbf{z}), $$
where $\mathbf{n}$ is the unit normal vector of $\mathcal{C}$.
For any $\mathbf{z}_0\in \mathcal{C}\backslash B_{r}$, where $r\ge r_1$ is a positive constant to be
determined later.
Let $\Omega$ be a open domain containing $0$ in $\mathbb{R}^{n+1}$
and $A$ be the second fundamental form of $\mathcal{C}$.
In a neighborhood of $\mathbf{z}_0$, we choose a local coordinate $F:\Omega\to\mathcal{C}$ satisfying
$$F(0)=\mathbf{z}_0,$$
$$\langle\partial_i F(0),\partial_j F(0) \rangle=\delta_{ij}$$
and
$$h_{ij}=A(\partial_i F,\partial_j F),$$
with $h_{ij}(0)=0$ if $i\ne j$.
Thus, in a neighborhood of $\mathbf{x}_0=\Pi(\mathbf{z}_0)$,
there exists a local coordinate $\tilde{F}:\Omega\to\Sigma$  given by
\begin{equation}
\tilde{F}(p)=\Pi\big(F(p)\big)=F(p)+u(p)\mathbf{n}(p).
\end{equation}
Here we identify $u(p)$ with $u(F(p))$, and identify $\mathbf{n}(p)$ with $\mathbf{n}(F(p))$.
Hence the tangent vectors $\partial_i\tilde{F}$,$i=1,\cdots,n$ of $\Sigma$ are given by
\begin{equation}
\begin{aligned}
\partial_i\tilde{F}=&d\Pi(\partial_i F)\\
=&\partial_i F+(\partial_iu) \mathbf{n} + u(\partial_i \mathbf{n}).
\end{aligned}
\end{equation}
Note that at $p=0$, $\partial_i \mathbf{n}(0)=-h_{ii}\partial_i F(0)$.
Thus it yields
\begin{equation}
\partial_i\tilde{F}=\big(1-h_{ii}\big)\partial_iF+\big(\partial_i u\big)\mathbf{n}.
\end{equation}
Therefore, at $p=0$, the metric  $g=(g_{ij})$ of $\Sigma$ is given by
\begin{equation}\label{eq2.k}
\begin{aligned}
g_{ij} =&\langle\partial_i\tilde{F},\partial_j\tilde{F}\rangle\\
=&(1-h_{ii}u)(1-h_{jj}u)\delta_{ij}+\partial_iu\partial_ju.
\end{aligned}
\end{equation}

According to Lemma \ref{lemma.B}, there exists a constant $C_1$ such that
\begin{equation} \label{eq2.l}
|u(\mathbf{z}_0)|\le C_1|\mathbf{z}_0|^{-1}
\end{equation}
and
\begin{equation} \label{eq2.l-a}
 |\nabla^\mathcal{C}u(\mathbf{z}_0)|\le C_1|\mathbf{z}_0|^{-2},
\end{equation}
where $\nabla^\mathcal{C}$ is the gradient operator on $\mathcal{C}$.
On the other hand, it is obvious that there exist two positive constants $r_2$ and $C_2$
such that for $\mathbf{z}_0\in \mathcal{C}\backslash B_{r_2}$,
\begin{equation}\label{eq2.m}
|A(\mathbf{z}_0)|\le C_2|\mathbf{z}_0|^{-1}.
\end{equation}
Using  \eqref{eq2.k}, \eqref{eq2.l},\eqref{eq2.l-a} and \eqref{eq2.m}, for any $0<\epsilon<1/2$,
we can choose a positive constant $r_3>\max\{r_1,r_2\}$
such that the metric $g=(g_{ij})$ and the inverse  $g^{-1}=(g^{ij})$, at $p=0$,
can be determined by the coefficients
\begin{equation} \label{eq2.n}
g_{ij}=\delta_{ij}+b_{ij}
\end{equation}
and
\begin{equation} \label{eq2.n-b}
 g^{ij}=\delta_{ij}+d_{ij},
\end{equation}
where $B=(b_{ij})$ and $ D=(d_{ij})$ are two small matrices with
$|b_{ij}|\le \epsilon,|d_{ij}|\le\epsilon$
and $1-\epsilon\le\det(g_{ij})\le 1+\epsilon$.

For any $f\in C^\infty (\Sigma)$, at point $\mathbf{x}_0=\Pi(\mathbf{z}_0)$, we have
\begin{equation}
\begin{aligned}
|\nabla f(\mathbf{x}_0)|^2 &=g^{ij}\partial_i\big(f\circ \tilde{F}(0)\big)\partial_j\big(f\circ \tilde{a}(0)\big) \\
&= \delta_{ij} \partial_i\big((f\circ \Pi) \circ F(0)\big)\partial_j\big((f\circ\Pi) \circ F(0)\big) \\
&\quad +d_{ij}\partial_i\big((f\circ \Pi) \circ F(0)\big)\partial_j\big((f\circ\Pi) \circ F(0)\big) \\
&\ge |\nabla^\mathcal{C}f\circ\Pi(\mathbf{z}_0)|^2 -\epsilon |\nabla^\mathcal{C}f\circ\Pi(\mathbf{z}_0)|^2 \\
&\ge \frac{1}{2} |\nabla^\mathcal{C}f\circ\Pi(\mathbf{z}_0)|^2 .
\end{aligned}
\end{equation}
On the other hand, due to \eqref{eq2.l}, we can choose $r_4\ge r_3+1$
such that for $\mathbf{x}_0\in \Sigma\backslash B_{r_4}$,
\begin{equation}
\big| |\Pi^{-1}(\mathbf{x}_0)|-|\mathbf{x}_0|\big|\le 1.
\end{equation}
Therefore, there exist two positive constants $C_3$ and $C_4$ such that for any $f\in C^\infty_0(\Sigma\backslash B_{r_4})$,
\begin{equation}
\int_{\Sigma\backslash B_{r_4}} |\nabla f(\mathbf{x})|^2 e^{-\frac{|\mathbf{x}|^2}{4}}d\mu
\ge C_3 \int_{\mathcal{C}\backslash B_{r_4-1}}|\nabla^\mathcal{C} f\circ\Pi(\mathbf{x})|^2 e^{-\frac{|\mathbf{x}|^2}{4}}d\mu
\end{equation}
and
\begin{equation}
\int_{\mathcal{C}\backslash B_{r_4-1}}|f\circ\Pi(\mathbf{x})|^2 e^{-\frac{|\mathbf{x}|^2}{4}}d\mu
\ge C_4 \int_{\Sigma\backslash B_{r_4}}|f(\mathbf{x})|^2 e^{-\frac{|\mathbf{x}|^2}{4}}d\mu .
\end{equation}
Applying Lemma \ref{lemma.A} with the constant $CC_3^{-1}C_4^{-1}$, we know that there exists a positive constant $r_5\ge r_4$ such that
for any $f\in C^\infty_0(\Sigma\backslash B_{r_5})$,
\begin{equation}
\begin{aligned}
\int_{\Sigma\backslash B_{r_5}} |\nabla f(\mathbf{x})|^2 e^{-\frac{|\mathbf{x}|^2}{4}}d\mu
&\ge C_3 \int_{\mathcal{C}\backslash B_{r_5-1}}|\nabla^\mathcal{C} f\circ\Pi(\mathbf{x})|^2 e^{-\frac{|\mathbf{x}|^2}{4}}d\mu \\
&\ge C_3(CC_3^{-1}C_4^{-1}) \int_{\mathcal{C}\backslash B_{r_5-1}}| f\circ\Pi(\mathbf{x})|^2 e^{-\frac{|\mathbf{x}|^2}{4}}d\mu \\
&\ge C_3(CC_3^{-1}C_4^{-1})C_4 \int_{\Sigma\backslash B_{r_5}}|f(\mathbf{x})|^2 e^{-\frac{|\mathbf{x}|^2}{4}}d\mu \\
& = C\int_{\Sigma\backslash B_{r_5}}|f(\mathbf{x})|^2 e^{-\frac{|\mathbf{x}|^2}{4}}d\mu.
\end{aligned}
\end{equation}
Hence \eqref{eq2.32} is true. This ends  the proof of Lemma \ref{lemma.C}.
\end{proof}

\begin{lemma} \label{lemma.E}
Let $\Sigma$ be an asymptotically cylindrical
end of a smooth properly embedded self-shrinker in $ \mathbb{R}^{n+1}$.
For any given positive constant $C$, there exists a positive constant
$r$ such that for all $f\in C^\infty_0\big((\mathbb{S}^k({\sqrt{2k}})\times \mathbb{R}^{n-k})\backslash B_{r}\big)$,
\begin{equation}\label{eq-lemma.E}
C\int_{\Sigma\backslash B_{r}} f^2(\mathbf{x}) e^{-\frac{|\mathbf{x}|^2}{4}}d\mu
	 \leq \int_{\Sigma\backslash B_{r}}|\nabla f(\mathbf{x})|^2 e^{-\frac{|\mathbf{x}|^2}{4}} d\mu,
\end{equation}
where $\nabla$ is the gradient operator on $\Sigma$.
\end{lemma}

\begin{proof}
We first claim that for given $\epsilon>0$, there exists $r_1$  such that, out of compact set $K$,
$\Sigma$ is given the the graph of a smooth function
$u(\cdot):\big(\mathbb{S}^k({\sqrt{2k}})\times \mathbb{R}^{n-k}\big)\backslash B_{r_1}\to\mathbb{R}$.
Moreover, for any $\mathbf{z}_0\in \big(\mathbb{S}^k({\sqrt{2k}})\times \mathbb{R}^{n-k}\big)\backslash B_{r_1}$,
$u$ satisfies
\begin{equation} \label{eq2.o}
|u(\mathbf{z}_0)|\leq \epsilon
\end{equation}
and
\begin{equation} \label{eq2.o-b}
 |\nabla^{\mathbb{S}^k({\sqrt{2k}})\times \mathbb{R}^{n-k}} u(\mathbf{z}_0)|\leq \epsilon,
\end{equation}
where $\nabla^{\mathbb{S}^k({\sqrt{2k}})\times \mathbb{R}^{n-k}}$ is the gradient operator on the generalized cylinder.
In fact, denote $\Sigma_t=\Sigma-t\xi$.
According to the definition of asymptotically cylindrical ends,
there exists a positive constant $T=T(\epsilon)$ such that if $t\ge T$.
Then there exists a subset
$\Omega_t \subset\Sigma\cap \big(B_{5}\backslash B_{1/5}\big)$
written as the graph of a function
$$\tilde{u}(\cdot,t):\big(\mathbb{S}^k({\sqrt{2k}})\times \mathbb{R}^{n-k}\big)\cap\big(B_{4}\backslash B_{1/4}\big)\to\mathbb{R}$$
satisfying
$$|\tilde{u}|\le\epsilon \quad \mbox{and} \quad
|\nabla^{\mathbb{S}^k({\sqrt{2k}})\times \mathbb{R}^{n-k}}\tilde{u}|\le \epsilon.$$
Set $r_1\ge T+1$.
Define
$u:\big(\mathbb{S}^k({\sqrt{2k}})\times \mathbb{R}^{n-k}\big)\backslash B_{r_1}\to\mathbb{R}$
by
\begin{align}
u(\mathbf{z}_0)=\tilde{u}(\mathbf{z}_0-k\xi,k),
\end{align}
where $k+1\le|\mathbf{z}_0|\le k+3$.
Since $\Sigma_t$ is the translation of $\Sigma$ in the $\xi$ direction,
 $u$ is well defined.
Moreover, \eqref{eq2.o} and \eqref{eq2.o-b} holds.

Then, using the same strategy used in Lemma \ref{lemma.C},
we can show that  $g_{ij}$, $|\mathbf{x}|$
and  $|\nabla f|$ on $\Sigma$ can be bounded by
the corresponding quantities on the generalized cylinder.
Therefore, according to Lemma \ref{lemma.D}, we know that \eqref{eq-lemma.E} holds.
This finishes the proof of Lemma \ref{lemma.E}.
\end{proof}

Recall that Fischer-Colbrie\cite{FC85} proved that minimal surfaces has finite Morse index
if and only if it is stable outside a compact set (cf. \cite{PRS08}).
We can prove the following similar result for self-shrinkers.

\begin{lemma} \label{lemma.F}
Let $\Sigma$ be a complete self-shrinker in $\mathbb{R}^{n+1}$. Then the following are equivalent:
\begin{enumerate}
\item $\Sigma$ has finite Morse index.
\item There exists a compact set $K$ such that $\Sigma\backslash K$ is stable.
\end{enumerate}
\end{lemma}

Finally we give the proof of Theorem 1.1 by using Lemmas \ref{lemma.C}-\ref{lemma.F}.

\vskip 3mm
{\noindent \bf Proof of Theorem \ref{thm.1}}\hspace{0.2cm}
According to the assumption,  there exists a positive constant $r_0$ such that $\Sigma\backslash B_{r_0}$
consists of finite ends $E_1, \cdots, E_N$,
and each $E_i$ are either asymptotically conical or asymptotically cylindrical.
It is obvious that for the asymptotically conical end $E_i$, the
squared norm of second fundamental form $|A(\mathbf{x})|^2\to 0$ as $\mathbf{x}\in E_i$ goes to infinite.
On the other hand, for asymptotically cylindrical end $E_j$,  the squared norm of second fundamental form $|A(\mathbf{x})|^2\to 1/2$ as $\mathbf{x}\in E_j$ goes to infinite.
Hence we can choose $r_0$ large enough such that $|A|^2\le 1$ on all $E_i$.

From Lemmas \ref{lemma.C} and  \ref{lemma.E}, we know that there exist $r_i, i=1,\cdots,N$ such that for any $f\in C^\infty_0(E_i\backslash B_{r_i})$,
\begin{align}
\int_{E_i\backslash B_{r_i}} |\nabla f|^2 e^{-\frac{|\mathbf{x}|^2}{4}} d\mu \ge
	2\int_{E_i\backslash B_{r_i}} |f|^2 e^{-\frac{|\mathbf{x}|^2}{4}} d\mu.
\end{align}
Therefore, we have
\begin{equation}
\begin{aligned} \label{eq3.1}
I(f,f)&= \int_{E_i\backslash B_{r_i}} |\nabla f|^2 e^{-\frac{|\mathbf{x}|^2}{4}} d\mu
	-\int_{E_i\backslash B_{r_i}} \big(|A|^2+ \frac{1}{2}\big) f^2 e^{-\frac{|\mathbf{x}|^2}{4}} d\mu \\
	&\ge \frac{1}{2}\int_{E_i\backslash B_{r_i}} |f|^2 e^{-\frac{|\mathbf{x}|^2}{4}} d\mu \\
    &\ge 0.
\end{aligned}
\end{equation}
Set $r=\max\{r_0,r_1,\cdots,r_N \}$. According to the definition of stability, and using \eqref{eq3.1},
we deduce that $\Sigma\backslash B_r$ is stable.
Then Lemma \ref{lemma.F} tells that $\Sigma$ has finite Morse index. This completes the proof of Theorem \ref{thm.1}.
$\hfill{} \Box$

\vskip 5mm
\begin{flushleft}
\textsc{Xu-Yong Jiang\\
School of Computer Science and Artificial Intelligence\\
Changzhou University\\
Changzhou 213164, P. R. China}\\
\textit{e-mail:} \verb"jiangxy_1@163.com"
\end{flushleft}

\begin{flushleft}
\textsc{He-Jun Sun\\
College of Science\\
Nanjing University of Science and Technology\\
Nanjing 210094, P. R. China}\\
\textit{e-mail:} \verb"hejunsun@njust.edu.cn"
\end{flushleft}

\begin{flushleft}
\textsc{Peibiao Zhao\\
College of Science\\
Nanjing University of Science and Technology\\
Nanjing 210094, P. R. China}\\
\textit{e-mail:} \verb"pbzhao@njust.edu.cn"
\end{flushleft}

\end{document}